\documentclass[11pt,reqno]{amsart}

\usepackage{amsmath}
\usepackage{amsfonts}
\usepackage{amssymb}
\usepackage{amsthm}
\usepackage{hyperref}
\usepackage{enumerate}
\usepackage{cleveref}

\usepackage{mathrsfs}
\usepackage[top=0.9in, bottom=1.15in, left=1.1in, right=1.1in]{geometry}
 \usepackage{hyperref}
\hypersetup{colorlinks=true,linkcolor=blue,citecolor=magenta}
\newtheorem{theorem}{Theorem}

\theoremstyle{remark}

\newtheorem*{remark}{Remark added in 2022}
\newtheorem*{remarkWZ}{Comment from Wadim Zudilin, 29 November 2022}

\Crefname{conjecture}{Conjecture}{Conjectures}

\theoremstyle{plain}

\theoremstyle{plain}

\allowdisplaybreaks

\author{Andrew V. Sills}
\address{Department of Mathematical Sciences\newline
Georgia Southern University\newline
Statesboro, Georgia 30458, U.S.A.}
\email{asills@georgiasouthern.edu}

\title[Rogers--Ramanujan Type Identities by Constant Terms]{Derivation of Identities of the Rogers--Ramanujan Type by the Method of Constant Terms}
\begin{document}
\date{Original: June 1994; annotated with some corrections on November 23, 2022}
\maketitle
\addtocounter{section}{-1}
\section{Explanation}
What follows is
a lightly edited version of
 the author's unpublished master's essay, submitted in partial fulfillment of the requirements of
the degree of Master of Arts at the Pennsylvania State University, dated June 1994,
written under the supervision of Professor George E. Andrews.  
It was retyped by the author on November 23, 2022.  Obvious typographical errors in the
original were corrected without comment; hopefully not too many new errors were introduced
during the retyping.
Explanatory text added by the author in 2022 is notated by \emph{Remark added in 2022}.
After the initial posting on the arXiv on November 29, 2022, the author received email 
from Wadim Zudilin and George Andrews, pointing out some typos and making some interesting
comments.  These comments have been incorporated in this revised submission to the arXiv.
The bibliography in this version is more extensive than that of the original.  

Lastly, in 1994, I neglected to
mention that $q$ is being treated as a formal variable throughout.

\section{Introduction}
In Chapter 4 of his monograph~\emph{$q$-series: Their Development and Application in Analysis,
Number Theory, Combinatorics, Physics, and Computer Algebra}~\cite{A86}, George Andrews showed
that ``[g]iven the Rogers--Ramanujan identities, \dots [certain] results become easy consequences of constant
term arguments.''  The results referred to are other series--product identities which are similar in form to the
Rogers--Ramanujan identiites.

The method of constant terms is executed as follows: one starts with a series involving powers of $q$ in the numerator
and $q$-factorials in the denominator and possibly the numerator.  Then the series is re\"expressed as the constant term
in a product of two series involving a new variable, $z$.  The $q$-binomial theorem or one of its corollaries is
invoked to convert the two series into products, which are then grouped in a new way, the $q$-binomial theorem and
one or two of its corollaries are then invoked once again, and a new series appears.  In the case of all of the identities
addressed by Andrews in his monograph~\cite{A86}, the new series produced was always easily recognizable as some
form of the Rogers--Ramanujan identities multiplied by an infinite product.  As we shall demonstrate here, sometimes an
unfamiliar series is derived via this method, and combining this information with established results, we generate new
series--product identities.

The following standard notation will be used throughout:

\[ (a)_n = (a;q)_n = (1-a)(1-aq)(1-aq^2)\cdots (1-aq^{n-1}) \]
\[ (a)_\infty = (a;q)_\infty = \lim_{n\to\infty} (a;q)_n \]

The following non-standard notation will also be used: $CT[X]$ will denote the constant term, that is,
the co\"efficient of $z^0$ in the series or infinite product $X$.

The following identities will be assumed:

\begin{align}
G(q) \equiv \sum_{n=0}^\infty \frac{q^{n^2}}{(q)_n} &= \prod_{n=1}^\infty \frac{1}{(1-q^{5n-4})(1-q^{5n-1})} \label{RR1}\\
H(q) \equiv \sum_{n=0}^\infty  \frac{q^{n^2+n}}{(q)_n}&= \prod_{n=1}^\infty \frac{1}{(1-q^{5n-3})(1-q^{5n-2})} \label{RR2}\\
\sum_{n=0}^\infty \frac{(a)_n t^n}{(q)_n} &= \frac{(at)_\infty}{(t)_\infty} \label{qBT}\\
\sum_{n=-\infty}^\infty \frac{ (-1)^n q^{\binom n2} t^n }{ (b)_n } &= \frac{ (t)_\infty (q/t)_\infty (q)_\infty  }{ (b/t)_\infty (b)_\infty }
\label{i4}\\
\sum_{n=0}^\infty \frac{q^{\binom n2} t^n}{(q)_n} &= (-t)_\infty \label{i5}\\
\sum_{n=0}^\infty \frac{ q^{n^2} }{(q^4;q^4)_n } &= \frac{ G(q)}{(-q^2;q^2)_\infty} \label{i6}\\
\sum_{n=0}^\infty \frac{ q^{n^2 + 2n}}{(q^4;q^4)_n} &= \frac{ H(q)}{(-q^2;q^2)_\infty} \label{i7}\\
\sum_{n=0}^\infty \frac{ q^{n^2} (-q;q^2)_n}{ (q^4;q^4)_n} &= \frac{ (q^3;q^6)^2 (q^6;q^6)_\infty (-q;q^2)_\infty } { (q^2;q^2)_\infty }
\label{i8}\\
\sum_{n=0}^\infty \frac{ q^{2n^2 - n} (-q;q^2)_n}{(q^2;q^2)_n (q^2;q^4)_n} & = (-q)_\infty \label{i9}\\
\sum_{n=0}^\infty \frac{ q^{n^2} (-q)_n} { (q;q^2)_{n+1} (q)_n } &= \frac{ (q^3;q^6)^2 (q^6;q^6)_\infty (-q)_\infty } { (q)_\infty }
\label{i10}
\end{align}

Equations~\eqref{RR1} and~\eqref{RR2} are the Rogers--Ramanujan identities.  
Equation~\eqref{qBT} is the $q$-analog of the binomial series~\cite[p. 17, Theorem 2.1]{A76}.
Equations~\eqref{i4}~\cite[p. 115, Eq. (C.2)]{A86} and~\eqref{i5}~\cite[p. 19, Eq. (2.2.6)]{A76} can be easily deduced from~\eqref{qBT}.
Equations~\eqref{i6} and~\eqref{i7} were originally proved by Rogers and reproved by Slater~\cite[p. 153--154, Eqns. (20) and (16) respectively]{S52}.
Equations~\eqref{i8},~\eqref{i9}, and~\eqref{i10} were proved by Slater~\cite[pp. 154, 157; Eqs. (25), (52), and (26) respectively]{S52}.

\begin{remark}
I should have given more precise references for the various identities due to L. J. Rogers.
The Rogers--Ramanujan identities~\eqref{RR1},~\eqref{RR2} first appeared in~\cite[p. 328 (2); p. 330 (2), resp.]{R94}.
Eq.~\eqref{i6} first appeared in~\cite[p. 330]{R94}, and~\eqref{i7} in~\cite[p. 331, above (7)]{R94}.
Additionally, Eq.~\eqref{i8} first appeared in Ramanujan's lost notebook~\cite[p. 85, Entry 4.2.7]{AB09}.
\end{remark}

\section{New proofs for old identities}
As demonstrated in the proofs of Theorems~\ref{thm1} and~\ref{thm2} below, when the method of constant terms is applied to certain ``mod 5'' identities of
the Rogers--Ramanujan type, we simply reprove known results, as in Andrews~\cite{A86}.

\begin{theorem} \label{thm1}
\begin{equation} \sum_{n=0}^\infty \frac{ (-1)^n q^{3n^2} }{(q^4;q^4)_n (-q;q^2)_n } = \prod_{n=1}^\infty \frac{1}{ (1+q^n)(1-q^{5n-4})(1-q^{5n-1})}.
 \label{i11} 
 \end{equation}
\end{theorem}
\begin{proof}
 \begin{align*}
      & \phantom{==} \sum_{n=0}^\infty \frac{(-1)^n q^{3n^2}}{ (q^4;q^4)_n (-q;q^2)_n} \\
      &= CT\left[ \sum_{n=-\infty}^\infty  \frac{ (-1)^n q^{n^2} z^n }{ (-q;q^2)_n} \sum_{m=0}^\infty \frac{z^{-m} q^{2m^2} }{ (q^4;q^4)_m  }  \right] \\
      &= CT\left[ \frac{ (zq;q^2)_\infty (z^{-1}q; q^2)_\infty (q^2;q^2)_\infty }{ (-z^{-1}q;q^2)_\infty (-q;q^2)_\infty } \times (-z^{-1} q^2; q^4)_\infty \right]
      \mbox{ (by~\eqref{i4} and~\eqref{qBT})} \\
      & = \frac{1}{(-q;q^2)_\infty} CT\left[ (zq;q^2)_\infty (z^{-1}q;q^2)_\infty (q^2;q^2)_\infty \times \frac{(-z^{-1}q^2;q^4)_\infty }{ (-z^{-1};q^2)_\infty} \right] \\
      & = \frac{1}{(-q;q^2)_\infty} CT\left[ (zq;q^2)_\infty (z^{-1}q;q^2)_\infty (q^2;q^2)_\infty \times \frac{1}{ (-z^{-1};q^4)_\infty} \right] \\
      & = \frac{1}{(-q;q^2)_\infty} CT\left[ \sum_{n=-\infty}^\infty (-1)^n q^{n^2} z^n \sum_{m=0}^\infty \frac{(-1)^m z^{-m} }{ (q^4;q^4)_m } \right] 
          \mbox{ (by~\eqref{i4} and~\eqref{qBT})} \\
      & = \frac{1}{(-q;q^2)_\infty} \sum_{n=0}^\infty \frac{ q^{n^2}}{(q^4;q^4)_n} \\
      & = \frac{1}{(-q;q^2)_\infty} \times \frac{G(q)}{(-q^2;q^2)_\infty} \mbox{  (by~\eqref{i6}) }\\
      & = \frac{G(q)}{(-q)_\infty} \\
      & = \prod_{n=1}^\infty \frac{1}{ (1+q^n)(1-q^{5n-4})(1-q^{5n-1})}.
 \end{align*}
\end{proof}

\begin{theorem} \label{thm2}
\begin{equation} \sum_{n=0}^\infty \frac{ (-1)^n q^{3n^2-2n} }{(q^4;q^4)_n (-q;q^2)_n } = \prod_{n=1}^\infty \frac{1}{ (1+q^n)(1-q^{5n-3})(1-q^{5n-2})}.
 \label{i12} 
 \end{equation}
\end{theorem}
\begin{proof}
 \begin{align*}
      & \phantom{==} \sum_{n=0}^\infty \frac{(-1)^n q^{3n^2-2n}}{ (q^4;q^4)_n (-q;q^2)_n} \\
      &= CT\left[ \sum_{n=-\infty}^\infty  \frac{ (-1)^n q^{n^2} z^n }{ (-q;q^2)_n} \sum_{m=0}^\infty \frac{ z^{-m} q^{2m^2-2m} }{ (q^4;q^4)_m  }  \right] \\
      &= CT\left[ \frac{ (zq;q^2)_\infty (z^{-1}q; q^2)_\infty (q^2;q^2)_\infty }{ (-z^{-1}q;q^2)_\infty (-q;q^2)_\infty } \times (-z^{-1}; q^4)_\infty \right]
      \mbox{ (by~\eqref{i4} and~\eqref{qBT})} \\
      & = \frac{1}{(-q;q^2)_\infty} CT\left[ (zq;q^2)_\infty (z^{-1}q;q^2)_\infty (q^2;q^2)_\infty \times \frac{1}{ (-z^{-1} q^2;q^4)_\infty} \right] \\
      & = \frac{1}{(-q;q^2)_\infty} CT\left[ \sum_{n=-\infty}^\infty (-1)^n q^{n^2} z^n \sum_{m=0}^\infty \frac{ (-1)^m z^{-m} q^{2m} }{ (q^4;q^4)_m} \right]
          \mbox{ (by~\eqref{i4} and~\eqref{qBT})} \\
      & = \frac{1}{(-q;q^2)_\infty} \sum_{n=0}^\infty \frac{ q^{n^2+2n}}{(q^4;q^4)_n} \\
      & = \frac{1}{(-q;q^2)_\infty} \times \frac{H(q)}{(-q^2;q^2)_\infty} \mbox{  (by~\eqref{i7}) }\\
      & = \frac{H(q)}{(-q)_\infty} \\
      & = \prod_{n=1}^\infty \frac{1}{ (1+q^n)(1-q^{5n-3})(1-q^{5n-2})}.
 \end{align*}
\end{proof} 

Equations~\eqref{i11} and~\eqref{i12} were published previously with different proofs~\cite[pp. 155-156, Eqs. (19) and (15) respectively]{S52},
so nothing new appeared as a result of employing the method of constant terms.
\begin{remark}
  Theorems~\ref{thm1} and~\ref{thm2}  are due to Rogers~\cite[p. 339, Ex. 2; p. 330 (5), respectively]{R94}.
\end{remark}

\section{A new series--product identity}
\begin{remark}
With the benefit of hindsight, I see now that the identity presented in Theorem~\ref{thm3} was not really new.
It can be shown to be equivalent to Slater~\cite[p. 154, Eq. (48)]{S52} with $q$ replaced by $-q$.  \end{remark}

In light of the proofs of Theorems~\ref{thm1} and~\ref{thm2}, and the eight identities of L. J. Rogers that Andrews
proves via constant terms~\cite[pp. 33--36]{A86}, one might wonder if all identities of the Rogers--Ramanujan type
are provable by this method.  However, when we look at ``mod 6'' identities, we see immediately that this is not always
the case.

\begin{theorem}\label{thm3}
  \begin{equation} \label{i13}
   \sum_{n=0}^\infty \frac{ (-1;q^2)_n q^{n^2+n}}{ (q^2;q^2)_n (-q;q^2)_n } =
  \prod_{n=1}^\infty \frac{ (1-q^{6n-3})^2 (1-q^{6n}) }{(1-q^{4n-2})(1-q^{2n})} .
  \end{equation}
\end{theorem}

\begin{proof}
 \begin{align}
      & \phantom{==} \frac{(q^3;q^3)_\infty^2 (q^6;q^6)_\infty (-q;q^2)_\infty  }{ (q^2;q^2)_\infty } \notag \\
      & =\sum_{n=0}^\infty \frac{ q^{n^2} (-q;q^2)_n }{ (q^4;q^4)_n}  \mbox{ (by~\eqref{i8})} \label{i14} \\
       & =\sum_{n=0}^\infty \frac{ q^{n^2} (-q;q^2)_n}{ (-q^2;q^2)_n (q^2;q^2)_n} \notag \\
      &= CT\left[ \sum_{n=-\infty}^\infty  \frac{  q^{n^2} z^n }{ (-q^2;q^2)_n} \sum_{m=0}^\infty \frac{ z^{-m} (-q;q^2)_m }{ (q^2;q^2)_m  }  \right] \notag\\
      &= CT\left[ \frac{ (-zq;q^2)_\infty (-z^{-1}q; q^2)_\infty (q^2;q^2)_\infty }{ (z^{-1}q;q^2)_\infty (-q^2;q^2)_\infty } 
               \times \frac{(-z^{-1} q; q^2)_\infty}{( z^{-1};q^2)_\infty } \right]
      \mbox{ (by~\eqref{i4} and~\eqref{qBT})} \notag \\
      & = \frac{1}{(-q^2;q^2)_\infty} CT\left[ (-zq;q^2)_\infty (-z^{-1}q;q^2)_\infty (q^2;q^2)_\infty \times \frac{(-z^{-1}q;q^2)_\infty }{ (-z^{-1}q;q^2)_\infty} 
        \times \frac{1}{(z^{-1};q^2)_\infty} \right] \notag \\
      & = \frac{1}{(-q^2; q^2)_\infty} CT\left[ \sum_{n=-\infty}^\infty q^{n^2} z^n \sum_{m=0}^\infty \frac{(-1;q^2)_m z^{-m} q^m }{ (q^2;q^2)_m }
         \sum_{r=0}^\infty \frac{z^{-r}}{(q^2;q^2)_r} \right] 
         \mbox{ (by~\eqref{i4} and~\eqref{qBT})} \notag\\
      & = \frac{1}{(-q^2;q^2)_\infty} \sum_{m,r\geq 0} \frac{ q^{(m+r)^2} (-1;q^2)_m }{(q^2;q^2)_m (q^2;q^2)_r} \notag\\
     & = \frac{1}{(-q^2;q^2)_\infty} \sum_{m=0}^\infty \frac{ q^{m^2+m} (-1;q^2)_m}{(q^2;q^2)_m} \sum_{r=0}^\infty \frac{ q^{r^2 + 2mr}  }{ (q^2;q^2)_r} \notag\\
     & = \frac{1}{(-q^2;q^2)_\infty} \sum_{m=0}^\infty \frac{ q^{m^2+m} (-1;q^2)_m}{(q^2;q^2_m} (-q^{2m+1};q^2)_\infty  \label{i15} \mbox{ (by~\eqref{i5})}\\
      & = \frac{(-q;q^2)_\infty}{ (-q^2;q^2)_\infty} \sum_{m=0}^\infty \frac{ q^{m^2+m} (-1;q^2)_m }{ (q^2;q^2)_m (-q;q^2)_m }. \label{i16}
 \end{align}
\end{proof} 

Here, the method of constant terms leads us from the series~\eqref{i14} to another series~\eqref{i16}, and~\eqref{i16} is not a straightforward restatement
of either of the Rogers--Ramanujan identities.  However, we could have proceeded from~\eqref{i14} to~\eqref{i16} using 
Heine's transformation~\cite[p. 19, Cor. 2.3]{A76}, which would have been easier than using constant terms.

\section{New double series--product identities}
What then, is the value of the method of constant terms?  The key lies in the step justifying equation~\eqref{i15} above.
We were able to collapse the double series into a single series above because of the particular exponents on $q$.
Many times this will not be the case, and we will be left with an ``irreducible'' double series.
Theorems~\ref{thm4} and~\ref{thm5} are examples of this situation.

\begin{theorem}\label{thm4}
\begin{equation} \label{i17}
  \sum_{m,r\geq 0} \frac{ q^{4m^2 + 4mr + 2r^2 - r} }{ (q^4;q^4)_m (q^2;q^2)_r } = \prod_{n=1}^\infty (1 + q^{2n-1}).
\end{equation}
\end{theorem}

\begin{proof}
 \begin{align*}
 (-q)_\infty     & = \sum_{n=0}^\infty \frac{q^{2n^2 - n } (-q;q^2)_n }{ (q^2;q^2)_n (q^2;q^4)_n} \mbox{ (by~\eqref{i9})} \\
      &= CT\left[ \sum_{n=-\infty}^\infty  \frac{ q^{2n^2-n} z^n }{ (q^2;q^4)_n} \sum_{m=0}^\infty \frac{ z^{-m} (-q;q^2)_m }{ (q^2;q^2)_m  }  \right] \\
      &= CT\left[ \frac{ (-zq;q^4)_\infty (-z^{-1}q^3; q^2)_\infty (q^4;q^4)_\infty }{ (-z^{-1}q;q^4)_\infty (q^2;q^4)_\infty } \times 
        \frac{ (-z^{-1} q; q^2)_\infty}{ (z^{-1};q^2)_\infty } \right]
      \mbox{ (by~\eqref{i4} and~\eqref{qBT})} \\
      & = \frac{1}{(q^2;q^4)_\infty} CT\left[ (-zq;q^4)_\infty (-z^{-1}q^3;q^4)_\infty (q^4;q^4)_\infty \times \frac{(-z^{-1}q^3;q^4)_\infty }{ (z^{-1};q^2)_\infty} \right] \\
      & = \frac{1}{(q^2;q^4)_\infty} CT\left[ \sum_{n=-\infty}^\infty (-1)^n q^{2n^2-n} z^n \sum_{m=0}^\infty \frac{ q^{2m^2+m} z^{-m} }{ (q^4;q^4)_m } 
         \sum_{r=0}^\infty \frac{z^{-r}}{ (q^2;q^2)_r }\right] 
          \mbox{ (by~\eqref{i4},~\eqref{i5}, and~\eqref{qBT})} \\
      & = \frac{1}{(q^2;q^4)_\infty} \sum_{m,r\geq 0} \frac{ q^{2(m+r)^2 - (m+r) + 2m^2 + m} }{ (q^2;q^2)_m (q^2;q^2)_r } \\
    & = \frac{1}{(q^2;q^4)_\infty} \sum_{m,r\geq 0} \frac{ q^{4m^2 + 4mr + 2r^2 - r} }{ (q^2;q^2)_m (q^2;q^2)_r }. \\
 \end{align*}
\end{proof} 

\begin{remark}
George Andrews~\cite{GAPC} communicated the following generalization of~\eqref{i17}.
\begin{equation} \label{gg}
   \sum_{m,r\geq 0} \frac{ t^{2m+r} q^{4m^2 + 4mr + 2r^2 - r} }{ (q^4;q^4)_m (q^2;q^2)_r } 
   = \prod_{n=1}^\infty (1 + tq^{2n-1}).
\end{equation}  Andrews' proof of~\eqref{gg} relies on the $q$-Chu-Vandermonde 
sum~\cite[p. 37, Eq. (3.3.10]{A76} and an identity of Euler~\cite[p. 19, Eq. (2.2.6)]{A76}.
\end{remark}

\begin{theorem}\label{thm5}
\begin{equation} \label{i18}
  \sum_{m,r\geq 0} \frac{ q^{2m^2 + 2mr + r^2} }{ (q^2;q^2)_m (q)_r } = \prod_{n=1}^\infty \frac{(1-q^{6n-3})^2(1-q^{6n}) }{1-q^n }.
\end{equation}
\end{theorem}

\begin{proof}
 \begin{align*}
      & \phantom{==} \frac{ (q^3;q^6)_\infty^2 (q^6;q^6)_\infty (-q)_\infty}{ (q)_\infty } \\    
      & = \sum_{n=0}^\infty \frac{q^{n^2} (-q)_n }{ (q;q^2)_{n+1} (q)_n} \mbox{ (by~\eqref{i10})} \\
      &= CT\left[ \frac{1}{1-q} \sum_{n=-\infty}^\infty  \frac{ q^{n^2} z^n }{ (q^3;q^2)_n} \sum_{m=0}^\infty \frac{ z^{-m} (-q)_m }{ (q)_m  }  \right] \\
      &= CT\left[ \frac{ (-zq;q^2)_\infty (-z^{-1}q; q^2)_\infty (q^2;q^2)_\infty }{ (-z^{-1}q^2;q^2)_\infty (q;q^2)_\infty } \times 
        \frac{ (-z^{-1} q)_\infty}{ (z^{-1})_\infty } \right]
      \mbox{ (by~\eqref{i4} and~\eqref{qBT})} \\
      & = \frac{1}{(q;q^2)_\infty} CT\left[ (-zq;q^2)_\infty (-z^{-1}q;q^2)_\infty (q^2;q^2)_\infty \times (-z^{-1}q;q^2)_\infty \times \frac{1}{ (z^{-1})_\infty } \right] \\
      & = \frac{1}{(q;q^2)_\infty} CT\left[ \sum_{n=-\infty}^\infty  q^{n^2} z^n \sum_{m=0}^\infty \frac{ q^{m^2} z^{-m} }{ (q^2;q^2)_m } 
         \sum_{r=0}^\infty \frac{z^{-r}}{ (q)_r }\right] 
          \mbox{ (by~\eqref{i4},~\eqref{i5}, and~\eqref{qBT})} \\
      & = \frac{1}{(q;q^2)_\infty} \sum_{m,r\geq 0} \frac{ q^{(m+r)^2 + m^2} }{ (q^2;q^2)_m (q)_r } \\
    & = \frac{1}{(q;q^2)_\infty} \sum_{m,r\geq 0} \frac{ q^{2m^2 + 2mr + r^2} }{ (q^2;q^2)_m (q)_r }. \\
 \end{align*}
\end{proof} 

\begin{remarkWZ}

[T]his is a particular case of Bressoud's identities~\cite{B79}; when you replace 2s by 3s on the left, you get a 
Kanade--Russell mod 9 (see a discussion at the end of my paper with Ali Uncu~\cite{UZ}). 
The very same identity is recently treated again as the constant term, but in a more sophisticated way! See ~\cite[Theorem 1.1]{LW}\dots. The techniques of using the integrals for the complex terms and combining those with some theorems from the $q$-Bible are nicely given by Hjalmar [Rosengren] in~\cite{HR}.
\end{remarkWZ}

One might be suspicious that the double series in~\eqref{i17} and~\eqref{i18} are actually collapsible 
into single series.  Just because~\eqref{i5} is not applicable to the double series in~\eqref{i17} and~\eqref{i18} as
it was in~\eqref{i15}, does not mean that there is not some other way to simplify the double series.
However, currently there is no known general method for simplifying arbitrary multiple $q$-series.
Andrews~\cite{A81} has dealt with several special cases, but these apparently do not apply here.
So, by referring to these double series as ``irreducible,'' I mean that they cannot be transformed into a single-fold
sum by any currently known method.

\section{Conclusion}
This study suggests several directions for further research.
One could, of course, apply the method of constant terms to any identity of the Rogers--Ramanujan type
(there are 130 such identities in Slater~\cite{S52}), to see what series are generated by the method.
In a more advanced study, one could attempt to identify which partitions are enumerated by these series.

\section*{Acknowledgments}
First and foremost, I owe a great debt of gratitude to George Andrews, who served as my graduate advisor,
and has continued to act as a mentor to me over these many years since.
Thanks to Hjalmar Rosengren, who encouraged me to post this old essay on the arXiv, as a result
of a three-way email conversation among George Andrews, Hjalmar, and myself.   Within hours of version 1 first
appearing on the arXiv, I received emails from Wadim Zudilin and George Andrews with interesting comments
and some corrections; many thanks to them as well.

\end{document}